\documentclass[11pt]{amsart}
\usepackage[utf8]{inputenc}
\usepackage{bm}
\usepackage{fontenc}
\usepackage{amsfonts}
\usepackage{amssymb}
\usepackage{amsmath}
\usepackage{amsthm}\usepackage{mathtools}
\usepackage{enumerate}
\usepackage{enumitem} 
\usepackage[pagebackref,colorlinks,linkcolor=blue,citecolor=blue,urlcolor=blue,hypertexnames=true]{hyperref}
\usepackage[pagebackref,hypertexnames=true]{hyperref}
\usepackage{mathrsfs} 
\usepackage{tikz}
\usetikzlibrary{calc}
\usepackage{marginnote}
\usepackage{xcolor}
\usepackage{soul}\usepackage{tikz}
\usepackage[all,cmtip]{xy}
	
\newcommand{\Z}{\mathbb{Z}}

\newcommand{\N}{\mathbb{N}}
\newcommand{\C}{\mathbb{C}}

\newcommand{\eps}{\varepsilon}

\newcommand{\Manoa}{M\=anoa}

\newcommand{\Hawaii}{Hawai\kern.05em`\kern.05em\relax i}

\newcommand{\SOTh}{\mathrm{SOT}\text{-}}

\newcommand{\cstu}{\mathrm{C}^*_u}

\newcommand{\roeq}{\mathrm{Q}^*_u}

\newtheorem*{rigprob*}{Rigidity Problem for uniform Roe Algebras}
\newtheorem*{rigprobcorona*}{Rigidity Problem for uniform Roe Coronas}

\newcommand{\cstar}{$\mathrm{C}^*$}

\newcommand{\cZ}{\mathcal{Z}}

\newcommand{\cP}{\mathcal{P}}
\newcommand{\bbN}{\mathbb{N}}
\newcommand{\cB}{\mathcal{B}}
\newcommand{\cK}{\mathcal{K}}

\newcommand{\bbR}{\mathbb{R}}

\newcommand{\R}{\mathbb{R}}

\numberwithin{equation}{section}

\newtheorem{theorem}{Theorem}[section]
\newtheorem*{theorem*}{Theorem}
\newtheorem{proposition}[theorem]{Proposition}

\newtheorem*{proposition*}{Proposition}
\newtheorem{lemma}[theorem]{Lemma}
\newtheorem*{lemma*}{Lemma}
\newtheorem{corollary}[theorem]{Corollary}
\newtheorem*{corollary*}{Corollar}

\newtheorem*{fact*}{Fact}
\theoremstyle{definition}
\newtheorem{definition}[theorem]{Definition}
\newtheorem*{definition*}{Definition}

\newtheorem*{acknowledgments}{Acknowledgments}
\newtheorem{claim}[theorem]{Claim}
\newtheorem*{claim*}{Claim}

\newtheorem{conjecture}[theorem]{Conjecture}
\newtheorem*{conjecture*}{Conjecture}

\theoremstyle{remark}
\newtheorem{example}[theorem]{Example}
\newtheorem*{example*}{Example}
\newtheorem{remark}[theorem]{Remark}
\newtheorem*{remark*}{Remark}

\newtheorem*{note*}{Note}
\newtheorem*{question*}{Question}

\DeclareMathOperator{\propg}{prop}

\DeclareMathOperator{\rank}{rank}

\newcommand{\ri}{\partial}

\usepackage{enumitem}

\numberwithin{equation}{section}

\newcommand{\barh}{\partial^h}

\begin{document}
 
\title[Coarse equivalence versus bijective coarse equivalence]{Coarse equivalence versus bijective coarse equivalence of expander graphs}%

\date{\today} 

\author[Baudier]{Florent P. Baudier}
\address[F. P. Baudier]{Texas A\&M University, Department of Mathematics, College Station, TX 77843-3368, USA} 
 \email{florent@tamu.edu}
 \urladdr{https://www.math.tamu.edu/~florent/}

\author[Braga]{Bruno M. Braga}
\address[B. M. Braga]{IMPA, Estrada Dona Castorina 110, 22460-320, Rio de Janeiro, Brazi}
\email{demendoncabraga@gmail.com}
\urladdr{https://sites.google.com/site/demendoncabraga}
 
\author[Farah]{Ilijas Farah}

\address[I. Farah]{Department of Mathematics and Statistics\\
York University\\
4700 Keele Street\\
North York, Ontario\\ Canada, M3J 1P3\\
and 
Matemati\v cki Institut SANU\\
Kneza Mihaila 36\\
11\,000 Beograd, p.p. 367\\
Serbia}
\email{ifarah@yorku.ca}
\urladdr{https://ifarah.mathstats.yorku.ca}

\author[Vignati]{Alessandro Vignati}
\address[A. Vignati]{
Institut de Math\'ematiques de Jussieu (IMJ-PRG)\\
Universit\'e Paris Cit\'e\\
B\^atiment Sophie Germain\\
8 Place Aur\'elie Nemours \\ 75013 Paris, France}
\email{alessandro.vignati@imj-prg.fr}
\urladdr{http://www.automorph.net/avignati}

\author[Willett]{Rufus Willett}
\address[R. Willett]{University of \Hawaii~at \Manoa, 2565 McCarthy Mall, Keller 401A, Honolulu, HI 96816, USA} 
\email{rufus@math.hawaii.edu}
\urladdr{https://math.hawaii.edu/~rufus/}
 
\maketitle

\begin{abstract}
We provide a characterization of when a coarse equivalence between  coarse disjoint unions of expander graphs is close to a bijective coarse equivalence. We use this to show that  if the uniform Roe algebras of coarse disjoint unions of expanders graphs are isomorphic, then the metric spaces must be bijectively coarsely equivalent.     
\end{abstract}

\section{Introduction}

This paper concerns equivalences of metric spaces which preserve their large scale geometry. Large scale geometry is the study of metric spaces by an observer positioned arbitrarily far from them. In this spirit, events happening at a bounded distance from each other are identified. Formally, we say that maps $f,g\colon X\to (Y,d_Y)$ from a set into  a metric space are \emph{close} if 
\[\sup_{x\in X}d_Y(f(x),g(x))<\infty.\]
The morphisms of interest here are given by     maps $f\colon (X,d_X)\to (Y,d_Y)$ between metric spaces which are \emph{coarse}, meaning that   for all $r>0$ there is $s>0$ such that 
\[d_X(x,z)<r\ \text{ implies }\ d_Y(f(x),f(z))<s.\]
We say that  $(X,d_X)$ and $(Y,d_Y)$ are \emph{coarsely equivalent} if there are coarse maps $f\colon X\to Y$ and $g\colon Y\to X$ whose compositions are close to the respective identity maps; $g$ is then called a \emph{coarse inverse} for $f$.  

It is clear from these definitions that if $N\subseteq X$ is a \emph{net}, i.e., an $\eps$-dense and $\delta$-separated subset of $X$ for some $\eps,\delta>0$, then the inclusion $N\hookrightarrow X $ is a coarse equivalence.  This observation allows researches in  large scale geometry  to focus on  discrete metric spaces.   The metric spaces considered in these notes will be not only discrete, but also \emph{uniformly locally finite} (abbreviated as u.l.f.\ from now on), i.e.,  for each $r>0$ there is a uniform bound on the cardinality of the balls of radius $r$ in our metric spaces.\footnote{Some authors refer to those spaces as metric spaces with \emph{bounded geometry}.} 

For the last three decades, understanding when the existence of a coarse equivalence implies the existence of a bijective coarse equivalence\footnote{Most of this research actually focused on the slightly different question of whether the existence of a so-called quasi-isometry implies the existence of a bi-Lipschitz equivalence.  However for many interesting classes of spaces, including those in Questions \eqref{qf2} and \eqref{qr2}, this is the same question: compare for example \cite[Corollary 1.4.14]{NowakYuBook}.} has been the object of intensive research in the field: we refer to \cite{GenevoisTessera22} for a historical discussion. To the best of our knowledge, this program was initiated by Gromov  in \cite[\S1.A$'$]{GromovBook1993CambridgePress}. Besides  explicitly proposing  this general question, Gromov also pointed out the complexity of solving such problems even for very concrete cases. For instance, Gromov asked the following:  
\begin{enumerate}
\item \label{qf2} Are the free groups $\mathbb  F_2$ and $\mathbb F_3$ bijectively coarsely equivalent?
\item \label{qr2} Are all separated nets in $\R^2$ bijectively coarsely equivalent?
\end{enumerate}
Question \eqref{qr2} has been solved positively in \cite[Corollary 1]{Papasoglu1995GeoDedicata} while Question \eqref{qf2} was negatively solved independently in \cite[Section 3 and Theorem 4.1]{McMullen1998GAFA} and \cite[Theorem 1.1]{BuragoKleiner1998GAFA}. Notice that, when $\bbR^2$ is replaced by an infinite dimensional Banach space,  Question \eqref{qr2} has a positive answer (see \cite[Proposition 4.2]{LindenstraussMatouskovaPreissIsraelJournal2000}).

A u.l.f.\ metric space $X$ is \emph{amenable} if for any $\epsilon>0$ and $r>0$ there exists a finite subset $A$ of $X$ such that if $\partial_r(A):=\{x\in X\mid 0<d(x,A)\leq r\}$, then
$$
|\partial_r(A)|< \epsilon|A|. 
$$
For non-amenable u.l.f.\ metric spaces, the problem was completely settled independently by V. Nekrashevych \cite{Nekrashevych98} and K. Whyte \cite[Theorems 1.1 and 1.4]{Whyte1999Duke}\footnote{In \cite[page 104]{Harpe2000book} Papasoglou is also credited with an independent proof of Theorem 1.1, although the cited paper \cite{Papasoglu1995GeoDedicata} only explicitly deals with trees.}.  The methods\footnote{The precise statements in the literature are more restrictive, but the proofs still work in this level of generality.  We give a proof in the stated level of generality as Corollary \ref{ThmBijCoarseNonAme} below.} of Nekrashevych and Whyte establish the following result:

\begin{theorem}\label{inab}
Let  $X$ and $Y$ be u.l.f.\ metric spaces and suppose $Y$ is non-amenable. Then any coarse equivalence between $X$ and $Y$ is close to a bijective coarse equivalence.    \label{WhyteIntro}
\end{theorem}

For amenable metric spaces, it is straightforward to see that this does not hold:  while $\{n^2\mid n\in\N\}\hookrightarrow\{n^2,n^2+1\mid n\in\N\} $ is a coarse equivalence, these two spaces are not bijectively coarsely equivalent. In fact, as shown in \cite[Theorem 3.5]{Dymarz2005ContempMath}, if $G$ is a finitely generated amenable group endowed with the word metric given by a finite generating subset and $H$ is a finite index proper subgroup of $G$, then the inclusion $H\hookrightarrow G$ is a coarse equivalence which is not close to any bijective coarse equivalence. 
 
We now describe the main results of our paper. The metric spaces of interest will be coarse disjoint unions of \emph{expander graphs}. Expander graphs are sequences of finite graphs whose cardinalities converge to infinity, and which have the following two competing properties: (1) they are fairly sparse (in terms of number of edges relative to the number of vertices), yet  (2) they are highly connected.  See Definitions \ref{def:expander1} and \ref{def:expander2} for details.  Using a probabilistic approach, it was shown in \cite{Pinsker} that expander graphs not only exist, but exist in great abundance; soon after that, their first   explicit construction was presented in  \cite{Margulis1973ExpanderConstruction}. Nowadays there are several known constructions of expander graphs, we mention here a few:  they can be  constructed via algebraic methods, e.g. using Kazhdan's property (T) (see for example \cite[\S 3]{DavidoffSarnakValette}, \cite[\S6.1]{BekkaDelaValette}, or \cite{Tao15}), or via combinatorial methods, (for example, using zig-zag products as in \cite{RVW02} and \cite{MendelNaor14}). Expander graphs are, by nature, apt to applications outside of mathematics such as to  neural networks, physics, and transportation systems. An astonishingly large number of areas of pure mathematics  are also connected to, or use, expander graphs. Important for us, and for this paper, is the field of operator algebras, where expander graphs have been used to provide counterexamples to the coarse Baum--Connes Conjecture (\cite[\S7]{HigsonLafforgueSkandalis2002GAFA}), 
and to study exotic properties of $\mathcal{B}(\ell_2)$ (\cite[\S24.4 and 24.5]{Pisier:2020aa})
(amongst other things).  For more background, we recommend the excellent survey \cite{Kowalski.Expanders} and book \cite{LubotzkyBook2010}.  

In order to do coarse geometry with expanders, one arranges a sequence of such graphs into a \emph{coarse disjoint union}: roughly, this means that one builds a metric space from the disjoint union of the sequence by spacing the individual graphs further and further apart (see Definition \ref{cdu def} for details).  The intuition is that studying the coarse geometry of a coarse disjoint union of a sequence of finite graphs is the same as studying the asymptotic geometry of the underlying sequence.  We should note that a coarse disjoint union of finite graphs is always amenable, so Theorem \ref{inab} does not apply to such spaces.


Our main result is as follows.  It should be thought of as an analog of Theorem \ref{WhyteIntro} showing that the conclusion holds under an obvious necessary condition.

\begin{theorem}\label{ThmBijCorExp}
 Let $X$ and $Y$ be coarse disjoint unions of expander graphs, say $X=\bigsqcup_nX_n$ and $Y=\bigsqcup_nY_n$. Let $f\colon X\to Y$ be a coarse equivalence. The following are equivalent.
 \begin{enumerate}
 \item\label{ThmBijCorExpItem1} The map $f$ is close to a bijective coarse equivalence.
 \item\label{ThmBijCorExpItem2} There are cofinite subsets $N,M$ of $\N$ and a bijection $i\colon N\to M$ such that 
 \[\sum_{n\in \N\setminus N}|X_n|=\sum_{n\in \N\setminus M}|Y_n|\quad \]
 and for all $n\in N$
 \[
 \ |X_n|=|Y_{i(n)}|\ \text{ and }\ f(X_n)\subseteq Y_{i(n)}.
 \]
 \end{enumerate} 
\end{theorem}

We do not know whether the conclusion of Theorem~\ref{ThmBijCorExp} can fail if $X$ and $Y$ are coarse disjoint unions of finite connected graphs (that are not necessarily expanders). 
The statement of Theorem \ref{ThmBijCorExp} is purely combinatorial/geometric, but the proof requires uniformly finite homology \cite{Block:1992qp} as exploited by Whyte in \cite{Whyte1999Duke}.

\subsection{Applications to rigidity of uniform Roe algebras}

We  give an application of Theorem \ref{inab} to rigidity of uniform Roe algebras.  Uniform Roe algebras are $\mathrm{C}^*$-algebras associated to u.l.f.\ metric spaces; prototypical versions were introduced by Roe in \cite{Roe:1988qy} for index-theoretic purposes.  

We quickly recall the definition. Let $(X,d_X)$ be a u.l.f.\ metric space. Then $\ell_2(X)$ denotes the (complex) Hilbert space of square-summable complex-valued functions on $X$ and $\cB(\ell_2(X))$ denotes the space of all bounded operators on $\ell_2(X)$. The standard orthonormal  basis of $\ell_2(X)$ is denoted by $(\delta_x)_{x\in X}$. The propagation of an operator $a\in\cB(\ell_2(X))$ is given by 
\[\propg(a)=\sup\{d_X(x,z)\mid \langle a\delta_x,\delta_z\rangle\neq 0\}\]
and the \emph{uniform Roe algebra of $X$}, denoted by $\cstu(X)$, is the norm closure of all operators in $\cB(\ell_2(X))$ with finite propagation.   

Uniform Roe algebras have found applications in index theory (e.g., \cite{Spakula:2009tg,Engel:2018vm}), \cstar-algebra theory (e.g., \cite{Rordam:2010kx,Li:2017ac}), single operator theory (for example \cite{Rabinovich:2004xe,Spakula:2014aa}), topological dynamics (e.g., \cite{Kellerhals:2013aa,Brodzki:2015kb}), and mathematical physics (e.g., \cite{Cedzich:2018wx,Kubota2017}).  At the core of the field, lie the so-called ``rigidity problems'' for uniform Roe algebras.  Roughly, these problems ask how much of the geometry of a u.l.f.\ space is remembered by its  uniform Roe algebra. For instance, it was open until very recently whether isomorphic uniform Roe algebras must have come from coarsely equivalent spaces; this has been recently shown to be true in  \cite[Theorem 1.2]{BaudierBragaFarahKhukhroVignatiWillett2021uRaRig}. However, the following stronger conjecture remains open: 

\begin{conjecture}\label{Conjecture}
Let $X$ and $Y$ be u.l.f.\ metric spaces. If $\cstu(X)$ and $\cstu(Y)$ are isomorphic, then $X$ and $Y$ are bijectively coarsely equivalent. 
\end{conjecture}

Note that the converse direction is trivially true.  Before the present work, Conjecture \ref{Conjecture}
was known to be true in two cases: if one\footnote{Both property A and non-amenability are coarse invariants, and therefore one of the spaces has the property in question only if both of them do.} of $X$ and $Y$ has Yu's property A (see \cite[Corollary 6.13]{WhiteWillett2017}), or if one of $X$ and $Y$ is non-amenable (this is a consequence of \cite[Theorem 1.2]{BaudierBragaFarahKhukhroVignatiWillett2021uRaRig} and Theorem \ref{inab}).  Hence, Conjecture~\ref{Conjecture} remains open only for amenable metric spaces without property A. The most prominent spaces in this class are precisely the coarse disjoint unions of expander graphs and we use Theorem \ref{ThmBijCorExp} to add them to this list in  \S\ref{sec:Roe}: 

\begin{theorem}\label{ThmIsoURAImpliesBijCoarseEquivExp}
 Let $X$ and $Y$ be coarse disjoint unions of expander graphs. If $\cstu(X)$ and $\cstu(Y)$ are isomorphic, then $X$ and $Y$ are bijectively coarsely equivalent. 
\end{theorem}

Theorem \ref{ThmIsoURAImpliesBijCoarseEquivExp} is not just a corollary of Theorem \ref{ThmBijCorExp}. In order to prove it, we must also obtain a tight result about the structure of isomorphisms between coarse disjoint unions: we refer the reader to Theorem \ref{T.|Xn|=|Yn|} for details. Here, we simply mention that the proof of our structural result  depends on an analysis of \emph{uniform Roe coronas}, i.e., the quotient of uniform Roe algebras by the ideal of compact operators, which we denote by $\roeq(X)$. Consequently,   the \emph{Higson corona} will also play an important role  since, as has been recently shown in \cite[Proposition 3.6]{BaudierBragaFarahVignatiWillett2023},  the center of $\roeq(X)$ is canonically isomorphic to the algebra of continuous functions on the Higson corona of $X$.

 \subsection*{Outline of the paper}
 
 In \S\ref{SectionUULHomology}, we  describe the uniformly finite homology groups of u.l.f.\ metric spaces and use the methods of \cite{Whyte1999Duke} to derive the  technical results we need; the main result is Theorem \ref{ThmThm76Whyte1}, which characterizes when  a uniformly finite-to-one coarse map $X\to Y$ is close to an injective coarse map, where $X$ is a u.l.f.\ space and $Y=\bigsqcup_n Y_n$ is a coarse disjoint union of $t$-connected finite metric spaces (not necessarily expander graphs).  The characterization we arrive at is purely geometric / combinatorial, and does not require uniformly finite homology for its statement; however, homological methods are important for the proof.   
 
 We then apply those results in \S\ref{SectionMainResults} to sequences of expander graphs and obtain Theorem \ref{ThmBijCorExp}, as well as a ``one-sided''  version of it for injective maps (Theorem \ref{ThmINJCorExp} below).  The results of \S\ref{SectionMainResults} do not require homological methods directly: instead they use the criteria established in Theorem \ref{ThmThm76Whyte1}.  Indeed, the proof is essentially combinatorial, with the simple observation of Lemma  \ref{LemmaEasy} being a key ingredient.  
 
 In \S\ref{sec:Roe} we apply our results to uniform Roe algebras.  The main result is Theorem \ref{T.|Xn|=|Yn|} which shows that an isomorphism between uniform Roe algebras of coarse disjoint unions automatically comes from a coarse equivalence satisfying condition \eqref{ThmBijCorExpItem2} from Theorem \ref{ThmBijCorExp}.  To establish Theorem \ref{T.|Xn|=|Yn|}, we use material on uniform Roe coronas as studied by three of the authors in \cite{BragaFarahVignati2018AdvMath}, and previous results of the authors on the Higson corona from \cite[Proposition 3.6]{BaudierBragaFarahVignatiWillett2023}.

 \subsection*{Graph conventions}
 
 We will work with graphs considered as metric spaces.  All graphs are simple, unweighted, and undirected. We equip the vertex set of a graph with the \emph{shortest path metric} whereby the distance between two vertices is the smallest number of edges in a path between them (and infinity if no such path exists).  We will typically identify the graph with its vertex set equipped with this metric.

\section{Uniformly finite homology}\label{SectionUULHomology}

In this section, we describe the homology groups introduced in \cite{Whyte1999Duke}, and then derive the technical lemmas needed for our main results.

 Given a u.l.f.\ metric space $X$, let $C_0^{\textrm{uf}}(X)$ denote the vector space of formal sums of the form 
\[
a=\sum_{x\in X}a_x\cdot x,
\]
where $(a_x)_{x\in X}\in \ell_\infty(X,\Z)$,\footnote{Here $\ell_\infty(S,\Z)$ refers to the ring of bounded functions from a set $S$ to the integers, with pointwise operations.} and let $C_1^{\textrm{uf}}(X)$ be the vector space of formal sums of the form 
\[
a=\sum_{(x,z)\in X^2}a_{x,z}\cdot (x,z),
\]
such that $(a_{x,z})_{(x,z)\in X^2}\in \ell_\infty(X^2,\Z)$ and $\sup\{d(x,z)\mid a_{x,z}\neq 0\}<\infty$. There is a canonical boundary map $ \barh \colon C_1^{\textrm{uf}}(X)\to C_0^{\textrm{uf}}(X)$ determined by letting 
\[
\barh (1\cdot (x,z))=1\cdot x-1\cdot z,
\]
for all $(x,z)\in X^2$, and extending $\barh $ to the whole of $C_1^{\textrm{uf}}(X)$ linearly. 

\begin{definition}
Let $X$ be a u.l.f.\ metric space. The \emph{uniformly finite homology group of $X$} if defined as 
\[
H_0^{\textrm{uf}}(X)=C_0^{\textrm{uf}}(X)/\barh(C_1^{\textrm{uf}}(X)).
\]
For an element $\phi\in C_0^{\textrm{uf}}(X)$, we write $[\phi]$ for the homology class in $H_0^{\textrm{uf}}(X)$ that it represents. For a subset $A$ of $X$, we write $[A]\in H_0^{\textrm{uf}}(X)$ for the homology class associated to $\sum_{x\in A}1\cdot x$.
\end{definition}
 
Uniformly finite homology groups were first introduced in \cite{Block:1992qp}. We also refer the reader to \cite[Section 7.2]{NowakYuBook} for a textbook exposition. 
 
A map $f\colon X\to Y$ is \emph{uniformly finite-to-one} if there is $C>0$ such that the cardinality of $f^{-1}(\{y\})$ is at most $C$ for all $y\in Y$. Notice that if $f\colon X\to Y$ is coarse and uniformly finite-to-one, then $f$ canonically induces a map $f_*\colon H_0^{\textrm{uf}}(X)\to H_0^{\textrm{uf}}(Y)$. Indeed, since $f$ is uniformly finite-to-one, the map 
\begin{align*}
\sum_{x\in X} a_x\cdot x\in C_0^{\textrm{uf}}(X) \mapsto \sum_{x\in X} a_x\cdot f(x)\in C_0^{\textrm{uf}}(Y)
\end{align*}
 is well-defined. Moreover, as $f$ is coarse, this map takes $ \barh(C_1^{\textrm{uf}}(X))$ into $ \barh(C_1^{\textrm{uf}}(Y))$. Therefore, $f$ induces the map 
\begin{align*}
f_*\colon H_0^{\textrm{uf}}(X)& \to H_0^{\textrm{uf}}(Y)\\
\Bigg[\sum_{x\in X} a_x\cdot x\Bigg]&\mapsto \Bigg[\sum_{x\in X} a_x\cdot f(x)\Bigg].
\end{align*} 
If $g\colon X\to Y$ is uniformly finite-to-one and close to $f$ (whence also coarse), it is straightforward to check that $f_*=g_*$. 

The main result of \cite{Whyte1999Duke} gives a complete characterization of when a coarse equivalence between u.l.f.\ metric spaces is close to a coarse equivalence which is also a bijection. Precisely:

\begin{theorem} \emph{(}\cite[Theorem 1.1]{Whyte1999Duke}\emph{).}
 Let $X$ and $Y$ be u.l.f.\ metric spaces and $f\colon X\to Y$. 
 \begin{enumerate}
 \item\label{Thm11WhyteItem1} If $f$ is coarse and uniformly finite-to-one, then $f$ is close to an injective coarse map if and only if there is $Z\subseteq Y$ such that $f_*([X])=[Z]$.
 \item\label{Thm11WhyteItem2}
If $f$ is a coarse equivalence, then $f$ is close to a bijective coarse equivalence if and only if $f_*([X])=[Y]$. 
\end{enumerate}\label{Thm11Whyte}
\end{theorem}

\begin{remark}
The statement of \cite[Theorem 1.1]{Whyte1999Duke} is slightly different from the one of Theorem \ref{Thm11Whyte}. Precisely, as stated, \cite[Theorem 1.1]{Whyte1999Duke} only deals with part \eqref{Thm11WhyteItem2} of Theorem \ref{Thm11Whyte} and, instead of coarse equivalences, the result is about quasi-isometries. However, its proof gives us the result stated above verbatim; we explain this here for completeness. The proof of \cite[Theorem 1.1]{Whyte1999Duke} is presented as follows: it shows that if $f\colon X\to Y$ is any uniformly finite-to-one coarse map such that $f_*([X])=[Y]$, then there is an injective map $g\colon X\to Y$ which is close to $f$ (this is \cite[Lemma 4.2]{Whyte1999Duke}). Being close to $f$, the coarseness of $f$ passes to $g$; thus we conclude part \eqref{Thm11WhyteItem1} of Theorem \ref{Thm11Whyte}. If we assume furthermore that $f$ is a quasi-isometry, then the proof of the Schr\"oder-Bernstein theorem obtained by K\"{o}nig (see \cite[Page 88]{HalmosBook}) allows Whyte to conclude that there must be a bijection $h\colon X\to Y$ which is also close to $f$. As $f$ is a quasi-isometry, so is $h$. If instead $f$ were only a coarse equivalence, the existence of $h$ is obtained in exactly the same way as in Whyte's proof and, being close to a coarse equivalence, $h$ must be a coarse equivalence also; this is part \eqref{Thm11WhyteItem2} of Theorem \ref{Thm11Whyte}.
\end{remark}

Whyte also provided a useful criterion for an element in $C^{\textrm{uf}}_0(X)$ to be a boundary. To state it, a definition is necessary: if $A$ is a subset of a metric space $(X,d_X)$ and $r>0$, we write \[\ri_r(A)\coloneqq \{x\in X\setminus A\mid d_X(x,A)\leq r\}\] for the \emph{(outer) $r$-boundary of $A$}. If $r=1$, we abbreviate this to $\ri(A)$. If $(X,d_X)$ is a graph equipped with its shortest path metric then $\ri(A)$ is the \emph{(outer) vertex-boundary} of $A$.
 
\begin{theorem} \emph{(}\cite[Theorem 7.6]{Whyte1999Duke}\emph{).}\label{T.[a]=0iff...}
 Let $X$ be a u.l.f.\ metric space and $a=\sum_{x\in X}a_x\cdot x$ be in $C_0^{\textrm{uf}}(X)$. Then, $[a]=0$ if and only if there are $t>0$ and $C>0$ such that 
 \[
 \Bigg|\sum_{x\in A}a_x\Bigg|\leq C|\ri_t(A)|
 \]
 for all finite $A\subseteq X$.
\label{Thm76Whyte}
\end{theorem}

We now isolate some consequences of Theorems \ref{Thm11Whyte} and \ref{Thm76Whyte}. Given $t>0$, recall that a metric space $(X,d_X)$ is called \emph{$t$-connected} if for all $x,y\in X$ there are $x_0,\ldots, x_n\in X$ such that $x=x_0$, $y=x_n$, and $d_X(x_{i-1},x_i)\leq t$ for all $i\in \{1,\ldots, n\}$. Connected graphs are examples of $1$-connected metric spaces. We also need to recall the following definition. 
 
\begin{definition}\label{cdu def}
If $(X_n)_n$ is a sequence of finite metric spaces, their \emph{coarse disjoint union} is the disjoint union $\bigsqcup_n X_n$ equipped with any\footnote{As there is a choice involved here, a coarse disjoint union is not uniquely defined. However, it is not difficult to see that any two choices of metrics satisfying these conditions will be (bijectively) coarsely equivalent via the identity map.} metric $d$ that restricts to the given metric on each $X_n$  and that satisfies $d(X_n,X_m)\to \infty$ as $(n,m)\to\infty$ in $\{(n,m)\in \N^2\mid n\neq m\}$.
\end{definition}

\begin{corollary}
 Let $t_0\geq 0$. Let $X$ be a u.l.f.\ metric space and assume that $X=\bigsqcup_n X_n$ is the coarse disjoint union of $t_0$-connected finite metric spaces. Let $a=\sum_{x\in X}a_x\cdot x$ be in $C_0^{\textrm{uf}}(X)$ be such that $\sum_{x\in X_n}a_x=0$ for all $n\in\N$. The following are equivalent:
\begin{enumerate}
\item \label{a=0} $[a]=0$. 
\item \label{exists t} There exist $t>0$, $C>0$, and $n_0\in\N$ such that 
 \[
 \Bigg|\sum_{x\in A}a_x\Bigg|\leq C|\ri_t(A)| 
 \]
 for all $n\geq n_0$, and all $A\subseteq X_n$.
 \end{enumerate}
 \label{CorThm76Whyte}
\end{corollary}

\begin{proof}
The implication \eqref{a=0} $\Rightarrow$ \eqref{exists t} is (a special case of) Theorem~\ref{T.[a]=0iff...}. 

We now show \eqref{exists t} $\Rightarrow$ \eqref{a=0}. Fix $t$, $C>0$ and $n_0\in\N$  as in the statement. Since $|\partial_s(A)|\leq|\partial_t(A)|$ whenever $s\leq t$, we can assume that $t\geq t_0$. As $X=\bigsqcup_nX_n$ is a coarse disjoint union, replacing $n_0$ by a larger number if necessary, we can assume furthermore that $\ri_t(X_n)=\emptyset$ for all $n\geq n_0$. We claim that there is $C'>0$ such that 
\begin{equation}\label{EqCorThm76Whyte1}
\Bigg|\sum_{x\in A}a_x\Bigg|\leq C'|\ri_{t}(A)|
\end{equation}
for all $A\subseteq \bigsqcup_{n=1}^{n_0}X_n$. Indeed, since $Z\coloneqq \bigsqcup_{n=1}^{n_0}X_n$ is finite, letting $C’\coloneqq \sum_{x\in Z} |a_x|$ 
we have that \eqref{EqCorThm76Whyte1} holds for all $A\subseteq Z$ such that $\ri_t(A)\neq\emptyset$. On the other hand, if $A\subseteq Z$ is such that $\ri_t(A)=\emptyset$, the fact that each $X_n$ is $t$-connected implies that $A\cap X_n$ is either $\emptyset$ or $X_n$ for each $n\leq n_0$. Therefore, since $\sum_{x\in X_n}a_x=0$ for all $n\in\N$, \eqref{EqCorThm76Whyte1} holds for all $A\subseteq Z$.
 
Let now $A\subseteq X$ be an arbitrary finite subset and let $A_n=A\cap X_n$ for all $n\in\N$. Let $A'=\bigsqcup_{n=1}^{n_0}A_n$ and note that as $\ri_t(X_n)=\emptyset$ for all $n\geq n_0$, we have that
 \[
 \ri_t(A)=\ri_t(A')\sqcup\bigsqcup_{n>n_0}\ri_t(A_n).
 \] 
 We can then conclude, using the assumption, that 
\begin{align*}
\Bigg|\sum_{x\in A}a_x\Bigg|& \leq \Bigg|\sum_{x\in A'}a_x\Bigg|+\sum_{n>n_0}\Bigg|\sum_{x\in A_n}a_x\Bigg|\\
&\leq C'|\ri_t(A')|+C\sum_{n>n_0}|\ri_t(A_n)|\\
&\leq \max\{C,C'\}\ri_t(A).
 \end{align*}
 Since this holds for all $A\subseteq X$, Theorem \ref{Thm76Whyte} implies $[a]=0$ and we are done.
\end{proof}

\begin{remark}\label{rem:graphsgivet1}
In the setting of Corollary~\ref{CorThm76Whyte}, if the $t_0$-connected spaces $X_n$ satisfy the additional condition that whenever $t>t_0$ there is $N>0$ such that \begin{equation}\label{EqBoundaryGrowth}|\partial_{t}(A)|\leq N|\partial_{t_0}(A)|\ \text{ for all }n\in\bbN \ \text{ and }\ A\subseteq X_n,\end{equation} we have that condition \eqref{exists t}  is equivalent to the stronger condition that for all $t\geq t_0$ there exist $C>0$, and $n_0\in\N$ such that 
 \[
 \Bigg|\sum_{x\in A}a_x\Bigg|\leq C|\ri_t(A)| 
 \]
 for all $n\geq n_0$, and all $A\subseteq X_n$.

Notice that condition \eqref{EqBoundaryGrowth}  is automatic for instance if  each $X_n$ is a graph with all vertices of degree at most $k$, endowed with the shortest path metric: one can take with $N=k^{t-t_0}$. Hence, if $X=\bigsqcup_n X_n$ is a coarse disjoint union of $1$-connected graphs with uniformly bounded vertex degrees, then condition~\eqref{a=0} of Corollary~\ref{CorThm76Whyte} is equivalent to  the existence of  $C>0$ and $n_0\in\N$ such that 
 \[
 \Bigg|\sum_{x\in A}a_x\Bigg|\leq C|\ri(A)| 
 \]
 for all $n\geq n_0$, and all $A\subseteq X_n$.

\end{remark}

Before stating the main result of this section, we need two elementary lemmas.  The first concerns uniformly finite homology.

\begin{lemma}\label{fin vanish inf}
Let $X$ be a u.l.f.\ metric space which is the coarse disjoint union of finite metric spaces, say $X=\bigsqcup_nX_n$, and let $a=\sum_{x\in X}a_x\cdot x$ be in $ C^{\textrm{uf}}_0(X)$. Then, $[a]=0$ implies that there is $n_0\in\N$ such that 
\begin{equation}\label{sums}
\sum_{n<n_0}\sum_{x\in X_n}a_x=0\ \text{ and }\ \sum_{x\in X_n}a_x=0\ \text{ for all }\ n\geq n_0.
\end{equation}
\end{lemma}

\begin{proof}
If $[a]=0$, then $a= \barh b$ for some $b=\sum_{(x,z)\in X^2}b_{x,z}\cdot (x,z)$ in $ C_1^{\textrm{uf}}(X)$. By the definition of $C_1^{\textrm{uf}}(X)$, $r= \sup\{d_X(x,z)\mid b_{x,z}\neq 0\}$ is finite. Hence, there is $n_0\in \N$ such that for all distinct $n,m\in\N$, with $n\geq n_0$, all $x\in X_m$, and all $z\in X_n$, we have $d_X(x,z)>r$. Therefore, if $b_{x,z}\neq 0$, this implies that either $x,z\in \bigsqcup_{n < n_0}X_n$ or there is $n\geq n_0$ such that $x,z\in X_n$. This shows that 
\[
 \barh\Bigg(\sum_{(x,z)\in(\bigsqcup_{n<n_0} X_n)^2}b_{x,z}\cdot(x,z)\Bigg)=\sum_{n<n_0}\sum_{x\in X_n}a_x\cdot x 
\]
and 
\[
 \barh \Bigg(\sum_{(x,z)\in X_n}b_{x,z}\cdot(x,z)\Bigg)=\sum_{x\in X_n}a_x\cdot x
\]
for all $n\geq n_0$. The result follows from the observation that each sum in line \eqref{sums} consists of a sum of terms of the form $b_{x,z}$, where each $b_{x,z}$ appears twice with opposite signs.
\end{proof}

The next elementary lemma does not depend on the uniformly finite homology methods detailed above.
 
\begin{lemma}
 Let $X$ and $Y$ be u.l.f.\ metric spaces and suppose $Y=\bigsqcup_nY_n$ is a coarse disjoint union of finite metric spaces. If $f\colon X\to Y$ is a uniformly finite-to-one coarse map that is close to an injective map, then there is $n_0\in \N$ such that 
 \[\sum_{n<n_0}|f^{-1}(Y_n)|\leq \sum_{n<n_0}|Y_n| \ \text{ and } \ |f^{-1}(Y_n)|\leq |Y_{n}|\ \text{ for all } \ n\geq n_0.\]
 \label{LemmaInjCoarseClose}
\end{lemma}

\begin{proof}
 Suppose that $g\colon X\to Y$ is an injective map which is close to $f$. As $g$ is close to $f$ there is $c\in(0,\infty)$ such that $\sup_{x\in f^{-1}(Y_n)}d_Y(g(x),Y_n)<c$. Since $Y=\bigsqcup_nY_n$ is a coarse disjoint union, there is $n_0\in\N$ such that \[ g(f^{-1}(Y_n))\subseteq Y_n \text{ for all }\ n\geq n_0,\] and necessarily \[ g\Bigg(\bigsqcup_{n<n_0}f^{-1}(Y_n)\Bigg)\subseteq \bigsqcup_{n<n_0}Y_n.\]
As $g$ is injective, the result follows. 
\end{proof}

The next result is our main technical result of this section. 

\begin{theorem}
Let $t_0>0$. Let $X$ and $Y$ be u.l.f.\ metric spaces and suppose $Y=\bigsqcup_nY_n$ is a coarse disjoint union of $t_0$-connected finite metric spaces. Let $f\colon X\to Y$ be a uniformly finite-to-one coarse map. The following are equivalent.
\begin{enumerate}
\item \label{CorThm76Whyte1Item1} The map $f$ is close to an injective coarse map.
\item \label{CorThm76Whyte1Item2} There are $Z\subseteq Y$, $t\geq t_0$, $C>0$, and $n_0\in\N$ such that 
 \[\sum_{n<n_0}|f^{-1}(Y_n)|\leq \sum_{n<n_0}|Y_n| \ \text{ and } \ |f^{-1}(Y_n)|\leq |Y_{n}|\ \text{ for all } \ n\geq n_0\] and such that 
\[||A\cap Z|-|f^{-1}(A)||\leq C|\ri_{t}(A)|\]
 for all $n\geq n_0$ and all $A\subseteq Y_n$.
\end{enumerate} 
 \label{ThmThm76Whyte1}
\end{theorem}

\begin{proof}
We start by establishing notation. Given $Z\subseteq Y$, let $a^Z\in C_0^{\textrm{uf}}(Y)$ be given by \[a^Z=\sum_{y\in Z}1\cdot y-\sum_{x\in X}1\cdot f(x)\] and observe that $[a^Z]=[Z]-f_*([X])$. Now, write $a^Z=\sum_{y\in Y}a^Z_y\cdot y$, i.e., 
\begin{equation*}\label{EqFormulaAy}
a^Z_y\coloneqq \left\{\begin{array}{ll}
1-|f^{-1}(\{y\})|,& \text{ if } y\in Z,\\ 
-|f^{-1}(\{y\})|,& \text{ if } y \not\in Z.
\end{array}\right.
\end{equation*}
Hence, for all $A\subseteq Y$, we have 
 \begin{equation}\label{EqInjZEq}
 ||A\cap Z|-|f^{-1}(A)||= \Bigg|\sum_{y\in A}a^Z_y\Bigg|.
\end{equation}

 \eqref{CorThm76Whyte1Item1}$\Rightarrow$\eqref{CorThm76Whyte1Item2}: Suppose $f$ is close to an injective coarse equivalence. Firstly, notice that the existence of $n_0\in \N$ such that \[\sum_{n<n_0}|f^{-1}(Y_n)|\leq \sum_{n<n_0}|Y_n| \ \text{ and } \ |f^{-1}(Y_n)|\leq |Y_{n}|\ \text{ for all } \ n\geq n_0\]
 is given by Lemma \ref{LemmaInjCoarseClose}. Fix such $n_0$ and notice that replacing $n_0$ by a larger natural does not change this property. We must now show the existence of $Z\subseteq Y$ and $C>0$ as in the statement.
 
As $f$ is close to an injective coarse equivalence, Theorem~\ref{Thm11Whyte}\eqref{Thm11WhyteItem1} gives $Z\subseteq Y$ such that $f_*([X])=[Z]$; the latter condition implies that $[a^Z]=0$. Lemma \ref{fin vanish inf} then gives $m\in\N$ such that 
\[
\sum_{n<m}\sum_{y\in Y_n}a^Z_y=0\ \text{ and }\ \sum_{y\in Y_n}a^Z_y=0\ \text{ for all }\ n\geq m.
\]
We can then apply Corollary \ref{CorThm76Whyte} and get  $t$, $C>0$ and $m_0\in \N$ such that 
 \[
 \Bigg|\sum_{y\in A}a^Z_y\Bigg|\leq C|\partial_{t}(A)|
 \]
 for all $n\geq m_0$ and all $A\subseteq Y_n$. Replacing $m_0$ by a larger natural $m_1$ such that $\partial_t(Y_n)=\emptyset$ whenever $n\geq m_1$, we obtain, using \eqref{EqInjZEq}, that 
 \[
 ||A\cap Z|-|f^{-1}(A)|| \leq C|\partial_{t}(A)|
 \]
for all $n\geq m_1$ and all $A\subseteq Y_n$. Replace $n_0$ by $\max \{n_0,m_1\}$ and we are done.
 
 \eqref{CorThm76Whyte1Item2}$\Rightarrow$\eqref{CorThm76Whyte1Item1}: Let $Z\subseteq X$, $t\geq t_0$, $C>0$, and $n_0\in\N$ be as in the statement. Replacing $n_0$ by a larger number if necessary, we can assume that 
 \begin{equation}\label{Eq08Jun20231}
 \partial_{t}(Y_n)=\emptyset\ \text{ for all } \ n\geq n_0.
\end{equation}
 Let \[X'=\bigsqcup_{n\geq n_0}f^{-1}(Y_n)\ \text{ and }\ Y'=\bigsqcup_{n\geq n_0}Y_n\] and notice that, replacing $Z$ by $Z\cap Y'$, we can also assume that $Z\subseteq Y'$. 

Write $a^Z=a^{(0)}+a^{(1)}$, where $a^{(0)}$ is supported in $Y\setminus Y'$, and $a^{(1)}$ is supported on $ Y'$, i.e., 
\begin{equation}\label{EqFormulaAy2}
a^{(0)}=\sum_{y\in Y\setminus Y'}a^Z_y\cdot y\ \text{ and } \ a^{(1)}=\sum_{y\in Y'}a_y^Z\cdot y.
\end{equation}
 By hypothesis, $|X\setminus X'|\leq |Y\setminus Y'|$. So, there is an injective map $f^{(0)}\colon X\setminus X' \to Y\setminus Y'$.  Since the domain of $f^{(0)}$ is finite, it is automatically coarse.  Hence, to see that $f$ is close to an injective coarse map, it is enough to show that there is an injective coarse map $f^{(1)}\colon X'\to Y'$ which is close to $f\restriction X'\colon X'\to Y'$. By Theorem~\ref{Thm11Whyte}\eqref{Thm11WhyteItem1}, this will be shown once we prove that $[a^{(1)}]=0$.

From now on, we consider $a^{(1)}$ as an element of $C_0^{\textrm{uf}}(Y')$ and show that $[a^{(1)}]=0$. It follows from the formula for $a^{(1)}$ in \eqref{EqFormulaAy2} and equation \eqref{EqInjZEq} that for all $n\geq n_0$ and all $A\subseteq Y_n$, we have
\begin{equation}\label{EqInjZEq2}
 ||A\cap Z|-|f^{-1}(A)||= \Bigg|\sum_{y\in A}a_y^Z\Bigg|.
\end{equation}
The hypothesis then implies
\[
\Bigg|\sum_{y\in A}a_y^Z\Bigg|\leq C|\partial_{t}(A)|
\]
for all $n\geq m_0$ and all $A\subseteq Y_n$. Hence, as each $Y_n$ is $t_0$-connected, in order to show that $[a^{(1)}]=0$, it is enough to show that $a^{(1)}$ satisfies condition \eqref{exists t}  of Corollary \ref{CorThm76Whyte}. Fix $n\geq n_0$. By \eqref{Eq08Jun20231}, we have $\partial_{t}(Y_n)=\emptyset$. Hence, the hypothesis and \eqref{EqInjZEq2} give 
\[\Bigg|\sum_{y\in Y_n}a^Z_y\Bigg|=||Y_n\cap Z|-|f^{-1}(Y_n)||=0.\]
We can then apply Corollary \ref{CorThm76Whyte} and obtain that $[a^{(1)}]=0$ as desired. 
 \end{proof}

\section{Coarse equivalence versus bijective coarse equivalence}\label{SectionMainResults}

\subsection{Expander graphs} In this subsection, we characterize when a coarse equivalence between coarse disjoint unions of expander graphs is close to a bijective coarse equivalence, obtaining Theorem \ref{ThmBijCorExp}. 


There are several variants of the following definition in the literature: the differences generally affect the precise constants involved, but not the qualitative behavior.

\begin{definition}\label{def:expander1}
Let $k\in\N$, let $X$ be a finite graph such that each vertex is incident to at most $k$ edges, and $h>0$. We say that $X$ is a \emph{$(k,h)$-expander graph} if \[|\partial( A)|\geq h\Big(1-\frac{|A|}{|X|}\Big)|A|\]
for all $A\subseteq X$.
\end{definition}

The definition above is a ``local'' one. We are actually interested in infinite metric spaces that are obtained by ``gluing'' countably many expander graphs together in an appropriate way, i.e., a coarse disjoint union of $(k,h)$-expanders in the sense of Definition \ref{cdu def} above. We now define the main class of metric spaces under consideration here. 
 
 \begin{definition}\label{def:expander2}
Let $(X_n)_n$ be a sequence of finite graphs. We say that $(X_n)_n$ is a \emph{sequence of expander graphs} if there are $k\in\N$ and $h>0$ such that each $X_n$ is a $(k,h)$-expander graph, and such that $|X_n|\to\infty$ as $n\to\infty$. If we let $X$ be their coarse disjoint union, we simply say that \emph{$X=\bigsqcup_nX_n$ is the coarse disjoint union of expander graphs}.
\end{definition}
 
In Definition~\ref{def:expander2}, as each vertex of each $X_n$ is incident to at most $k$ edges, the space $X$ is clearly a u.l.f.\ metric space. 
 
Besides the results on uniformly finite homology of the previous section, the following combinatorial lemma is another key ingredient for our proof of Theorem \ref{ThmBijCorExp}.

 \begin{lemma}\label{LemmaEasy}
Let $X$ and $Y$ be finite sets, with $|X|=|Y|$, and $f\colon X\to Y$. Then, for all $A\subseteq Y$, we have that 
\[
||A|-|f^{-1}(A)||\leq \Big(\max_{y\in Y}|f^{-1}(\{y\})|-1\Big)\min\{|A|,|Y\setminus A|\}.
\]
\end{lemma}

\begin{proof}
 Let $m:=\max_{y\in Y}|f^{-1}(\{y\})|$. We may assume that $m>1$: otherwise, $f$ is a bijection by the pigeonhole principle and so $||A|-|f^{-1}(A)||=0$ for all $A\subseteq Y$, as $X$ and $Y$ are finite and of the same size. The assumption implies that for any $A\subseteq Y$, $0\leq |f^{-1}(A)|\leq m|A|$. Hence 
\[
-|A|\leq |f^{-1}(A)|-|A|\leq m|A|-|A|.
\]
As $m>1$, this implies that 
\begin{equation}\label{with a}
||A|-|f^{-1}(A)||\leq (m-1)|A|.
\end{equation}
The same reasoning with $A$ replaced by $Y\setminus A$ implies that 
\begin{equation}\label{with x-a}
||Y\setminus A|-|f^{-1}(Y\setminus A)||\leq (m-1)|Y\setminus A|.
\end{equation}
On the other hand, as $|Y\setminus A|=|Y|-|A|$ and as 
\[
|f^{-1}(Y\setminus A)|=|f^{-1}(Y)|-|f^{-1}(A)|=|X|-|f^{-1}(A)|=|Y|-|f^{-1}(A)|
\]
we have that 
\[
||Y\setminus A|-|f^{-1}(Y\setminus A)||=||A|-|f^{-1}(A)||.
\]
Combining this with line \eqref{with x-a} gives $||A|-|f^{-1}(A)||\leq (m-1)|Y\setminus A|$, and combining that with line \eqref{with a}, we are done.
 \end{proof}

 Before proving Theorem \ref{ThmBijCorExp}, we start proving the following ``injective version'' of it.
 
\begin{theorem}\label{ThmINJCorExp}
 Let $X$ and $Y$ be u.l.f.\ metric spaces and suppose $Y=\bigsqcup_nY_n$ is a coarse disjoint union of expander graphs. Let $f\colon X\to Y$ be a uniformly finite-to-one coarse map. The following are equivalent.
\begin{enumerate}
\item \label{ThmINJCorExpItem1} The map $f$ is close to an injective coarse map.
\item \label{ThmINJCorExpItem2} There is $n_0\in\N$ such that 
 \[\sum_{n<n_0}|f^{-1}(Y_n)|\leq \sum_{n<n_0}|Y_n| \ \text{ and } \ |f^{-1}(Y_n)|\leq |Y_{n}|\ \text{ for all } \ n\geq n_0.\]
 \end{enumerate}
\end{theorem}

 \begin{proof}
 The implication \eqref{ThmINJCorExpItem1}$\Rightarrow$\eqref{ThmINJCorExpItem2} follows immediately from Lemma \ref{LemmaInjCoarseClose} (and does not need any expander assumptions). So, we only show \eqref{ThmINJCorExpItem2}$\Rightarrow$\eqref{ThmINJCorExpItem1}. For that, let $n_0\in\N$ be as in the statement.

 As $(Y_n)_n$ is a sequence of expander graphs, there are $k\in\N$ and $h>0$ such that each $Y_n$ is a graph with all vertices having degree at most $k$ and such that
\[|\partial (A)|\geq h|A|\]
for all $n\in\N$ and all $A\subseteq Y_n$ with $|A|\leq |Y_n|/2$. As $f$ is uniformly finite-to-one, there is $m\in \N$ such that $|f^{-1}(\{y\})|\leq m$ for all $y\in Y$.

Suppose towards a contradiction that $f\colon X\to Y$ is not close to any injective  coarse map. Then, by Theorem \ref{ThmThm76Whyte1} and Remark \ref{rem:graphsgivet1}, for all $Z\subseteq Y$ and all $C>0$, we can pick $n\geq n_0$ and $A\subseteq Y_n$ such that 
\begin{equation}\label{EqThmBijCoarExp1}
|\ri(A)|< \frac{1}{C}||A\cap Z|-|f^{-1}(A)||.
\end{equation}
As $|f^{-1}(Y_n)|\leq |Y_{n}|$ and $f(f^{-1}(Y_n))\subseteq Y_{n}$ for all $n\geq n_0$,   we can pick $Z_n\subseteq Y_n$ such that $|f^{-1}(Y_n)|=|Z_n|$ and $f(f^{-1}(Y_n))\subseteq Z_n$ for any such $n$. Let $Z=\bigsqcup_{n\geq n_0} Z_n$ and  $C>0$ be such that $hC>k(m-1)$. Let now   $n\geq n_0$ and $A\subseteq Y_n$ be as in \eqref{EqThmBijCoarExp1}  for $Z$ and $C$. As $|f^{-1}(Y_n)|=|Z_{n}|$ and $f(f^{-1}(Y_n))\subseteq Z_{n}$, Lemma \ref{LemmaEasy} implies that 
\begin{align*}||A\cap Z|-|f^{-1}(A)||& \leq (m-1)\min\{|A\cap Z|,|Z_n\setminus A|\}\\
& \leq (m-1)\min\{|A|,|Y_n\setminus A|\}.
\end{align*}
Therefore, this together with   \eqref{EqThmBijCoarExp1} give
\begin{equation}\label{EqUsingLemmaEasy}
|\ri( A)|\leq \frac{m-1}{C}\min\{|A|,|Y_{n}\setminus A|\}.
\end{equation}

We can now finish the proof. Suppose $|A|\leq |Y_{n}|/2$. In this case, our choice of $h$ guarantees that \[|\ri( A)|\geq h|A|.\] Together with \eqref{EqUsingLemmaEasy}, this implies that $hC\leq m-1$, which contradicts our choice of $C$. We must then have that $|A|>|Y_{n}|/2$. But then $|Y_{n}\setminus A|\leq |Y_n|/2$ and our choice of $h$ implies that 
\[|\partial (Y_{n}\setminus A)|\geq h|Y_n\setminus A|.\] 
As each vertex in $Y_{n}$ has degree at most $k$ we have that $|\partial (Y_{n}\setminus A)|\leq k|\partial(A)|$ and so 
\[
|\partial (A)|\geq \frac{h}{k}|Y_{n}\setminus A|.
\]
Together with \eqref{EqUsingLemmaEasy}, this gives us that $hC\leq k(m-1)$, which contradicts our choice of $C$ once again. 
 \end{proof}

We isolate as a lemma an argument that has already been used many times in the literature (e.g., \cite[page 103]{Whyte1999Duke}, \cite[Corollary 6.10]{WhiteWillett2017}, \cite[IV.B-46]{Harpe2000book}).

\begin{lemma}
\label{LemmaSchroederBersteinKonig}
Let $X$ and $Y$ be u.l.f.\ metric spaces and $f\colon X\to Y$ be a coarse map. If there are injective coarse maps $g\colon X\to Y$ and $h\colon Y\to X$ which are close to $f$ and its coarse inverses, respectively, then $X$ and $Y$ are bijectively coarsely equivalent. 
\end{lemma}

\begin{proof}
As explained in \cite[Page 103]{Whyte1999Duke}, \cite[Corollary 6.10]{WhiteWillett2017}, or \cite[IV.B-46]{Harpe2000book}, this is an immediate consequence of K\"{o}nig's proof of the Schr\"oder-Bernstein theorem as exposited for example in \cite[Page 88]{HalmosBook}.
\end{proof}

\begin{proof}[Proof of Theorem \ref{ThmBijCorExp}]
 \eqref{ThmBijCorExpItem1}$\Rightarrow$\eqref{ThmBijCorExpItem2}: This implication follows similarly to the proof of Lemma \ref{LemmaInjCoarseClose}. Precisely, let $g\colon X\to Y$ be a bijective coarse equivalence which is close to $f$. As each $X_n$ is $1$-connected, and as $d_Y(Y_n,Y_m)\to \infty$ as $n+m\to \infty $ with $n\neq m$, there is $n_0\in \N$ such that for all $n\geq n_0$ there is $i(n)\in \N$ such that $g(X_n)\subseteq Y_{i(n)}$. Since $f$ is close to $g$, replacing $n_0$ by a larger number if necessary, we can also assume that $f(X_n)\subseteq Y_{i(n)}$ for all $n\geq n_0$.

As $g$ is surjective and as each $X_n$ is finite, the set $\{i(n)\in\N\mid n\geq n_0\}$ must be cofinite in $\N$. Therefore, applying the argument above for $g^{-1}$ and replacing $n_0$ by a larger number if necessary, the fact that $g$ is a bijection implies that $g(X_n)=Y_{i(n)}$ for all $n\in\N$. In particular, we conclude that $i$ defines a bijection between
\[
N=\{n\in\N\mid n\geq n_0\}\ \text{ and }\ M=\{i(n)\in\N\mid n\geq n_0\}
\]
and that 
\[
|X_n|=|Y_{i(n)}| \ \text{ for all }\ n\in N.
\]
Since $g$ restricts to a bijection between $X'=\bigsqcup_{n\in N}X_n$ and $Y'=\bigsqcup_{n\in M}Y_n$, $g$ must also restrict to a bijection between $X\setminus X'$ and $Y\setminus Y'$. Therefore,
\[
\sum_{n\in \N\setminus N}|X_n|=\sum_{n\in \N\setminus M}|Y_n|.
\]

\eqref{ThmBijCorExpItem2}$\Rightarrow$\eqref{ThmBijCorExpItem1}: There are two approaches to proving this implication. One of them is to notice that the proof follows completely analogously to the proofs of Theorems \ref{ThmThm76Whyte1} and \ref{ThmINJCorExp} with the only difference that instead of using Theorem \ref{Thm11Whyte}\eqref{Thm11WhyteItem1} in it, we replace it by Theorem \ref{Thm11Whyte}\eqref{Thm11WhyteItem2}. Alternatively, Theorem \ref{ThmINJCorExp} gives us injective coarse maps $g\colon X\to Y$ and $h\colon Y\to X$ which are close to $f$ and to any coarse inverse of it, respectively. The result then follows from Lemma~\ref{LemmaSchroederBersteinKonig}.
\end{proof}

\subsection{Application to $k$-stacking of expanders}

This section contains an application of the results in our previous section. As a special case, we obtain a sort of ``rigidity'' result for bipartite expander graphs (see Example \ref{bip ex} and Corollary \ref{CorStacking}).\footnote{Recall, a graph $X$ is \emph{bipartite} if there is a partition $X=V_1\sqcup V_2$ such that every edge of $X$ has one of its endpoints in $V_1$ and the other in $V_2$.} 

For that, we introduce the following definition: 

\begin{definition}\label{DefiStacking}
 Let $(X,d_X)$ be a metric space and $k\in\N$. We say that a metric space $(\bar X,d_{\bar X})$ is a \emph{$k$-stacking of $X$} if 
 \begin{enumerate}
 \item $\bar X=X\times \{1,\ldots, k\}$,
 \item the map $x\in X \mapsto (x,1)\in {\bar X}$ is a coarse equivalence, and 
 \item $\sup_{x\in X}\max_{i,j\leq k}d_{\bar X}((x,i),(x,j))<\infty$.
 \end{enumerate} 
\end{definition}

\begin{example}\label{bip ex}
 Given a metric space $(X,d_X)$, we can construct a $2$-stacking of $X$ by letting $d_{\bar X}$ be the metric in $\bar X=X\times \{1,2\}$ given by 
 \[d_{\bar X}((x,i),(z,j))=\left\{\begin{array}{ll}
 d_X(x,z), & \text{ if } i=j, \\
 d_X(x,z)+1, & \text{ if } i\neq j.
 \end{array}\right.\]
\end{example}

The motivation for Definition \ref{DefiStacking} comes from bipartite expander graphs: 

\begin{definition}
Let $k\in\N$ and $X$ be a graph such that each vertex is incident to at most $k$ edges. Suppose moreover that $X$ is bipartite graph with bipartition $X=V_1\sqcup V_2$ satisfying $|V_1|=|V_2|$. Given $h>0$, we say that $X$ is a \emph{$(k,h)$-bipartite expander} if for all $A\subseteq V_1$ with $|A|\leq |V_1|/2$ we have 
\[
|\partial A|\geq (1+h)|A|.
\]
\end{definition}

Expander graphs and their bipartite analogs are well-known to be in correspondence (see 
\cite[Remark 1.1.2.(ii)]{LubotzkyBook2010}). We now recall how a $(k,h)$-expander canonically generates a bipartitite expander, which is in particular a $2$-stacking of it. 

\begin{example}
Let $k\in\N$, $h>0$, and $X$ be a $(k,h)$-expander graph. The \emph{bipartite expander of $X$} is the $2$-stacking of $X$ defined as follows: let $\bar X=X\times \{1,2\}$ and define a graph structure on $\bar X$ by connecting each of the vertices $(x,1)$ to $(x,2)$ and to all $(z,2)$ such that $(x,z)$ is an edge of $X$. This makes $\bar X$ into a bipartite graph with bipartition $(X\times \{1\})\sqcup (X\times \{2\})$. Moreover, it is immediate to check that $\bar X$ is a $(k+1,h)$-bipartite expander. 
\end{example}

\begin{corollary}\label{CorStacking}
 Let $X$ and $Y$ be coarse disjoint union of expanders. Given $k\in\N$, let $\bar X$ and $\bar Z$ be $k$-stackings of $X$ and $Y$, respectively. The following are equivalent.
 \begin{enumerate}
 \item\label{CorStackingItem1} $X$ and $Y$ are bijectively coarsely equivalent.
 \item \label{CorStackingItem2} $\bar X$ and $\bar Y$ are bijectively coarsely equivalent.
 \end{enumerate}
\end{corollary}

\begin{proof}
\eqref{CorStackingItem1}$\Rightarrow$\eqref{CorStackingItem2}: This implication is immediate. 

\eqref{CorStackingItem2}$\Rightarrow$\eqref{CorStackingItem1}: Let $f\colon \bar X\to \bar Y$ be a bijective coarse equivalence. For each $n\in\N$, let $\bar X_n=X_n\times \{1,\ldots, k\}$ and $\bar Y_n=Y_n\times \{1,\ldots, k\}$. So $\bar X_n$ and $\bar Y_n$ can be seen canonically as subspaces of $\bar X$ and $\bar Y$, respectively, and $\bar X=\bigsqcup_n\bar X_n$ and $\bar Y=\bigsqcup_n\bar Y_n$ are coarse disjoint unions. The implication \eqref{ThmBijCorExpItem1}$\Rightarrow$\eqref{ThmBijCorExpItem2} of Theorem~\ref{ThmBijCorExp} gives cofinite subsets $N,M\subseteq \N$ and a bijection $i\colon N\to M$ such that 
\[
\sum_{n\in \N\setminus N}|\bar X_n|=\sum_{n\in \N\setminus M}|\bar Y_n| \quad \text{and} \quad f\Bigg(\bigsqcup_{n\in \N\setminus N} \bar X_n\Bigg)\subseteq \bigsqcup_{n\in \N\setminus M} \bar Y_n 
\]
 and such that 
 \[
 \ |\bar X_n|=|\bar Y_{i(n)}|\ \text{ and }\ f(\bar X_n)\subseteq \bar Y_{i(n)}
\ \text{ for all }\ n\in N.
\]
Consequently, we must have 
\[
\sum_{n\in \N\setminus N}| X_n|=\sum_{n\in \N\setminus M}| Y_n|\ \text{ and }\ | X_n|=| Y_{i(n)}|\ \text{ for all }\ n\in N.
 \]
Let now $j\colon X\to X\times \{1\}\subseteq \bar X$ be the canonical inclusion and $\pi\colon \bar Y\to Y$ be the canonical projection. Then, letting $g=\pi\circ f\circ j$, we have that $g$ is a coarse equivalence between $X$ and $Y$ such that 
\[
g(X_n)\subseteq Y_{i(n)}\ \text{ for all }\ n\in A.
\]
The result then follows from the implication \eqref{ThmBijCorExpItem2}$\Rightarrow$\eqref{ThmBijCorExpItem1} of Theorem
\ref{ThmBijCorExp}.
\end{proof}
 
\subsection{Coarse bijective equivalences and non-amenability of metric spaces} 
In this section, we provide an elementary proof of Theorem \ref{WhyteIntro}.  After (re)discovering this simple argument, we learned that a similar elementary proof of Theorem \ref{WhyteIntro} is due to V. Nekrashevych, from his 1998 PhD thesis written in Ukrainian, and can be found in \cite[IV.B-46]{Harpe2000book}.  We nonetheless present our proof here, partly as Proposition \ref{PropCoarseEmbImpliesInjCNonAme} is not explicitly in the literature, and seems likely to be useful in other contexts.

One of the many equivalent characterizations of non-amenability is the following (cf. \cite[Lemma 2.1]{Whyte1999Duke}) 

\begin{definition}
Let $(X,d_X)$ be a u.l.f.\ metric space. For each $r>0$ and $S\subseteq X$, the \emph{$r$-neighborhood of $S$} is given by 
\[
N_r(S)\coloneqq \{x\in X\mid d(x,S)\le r\}.
\]
We say that $X$ is \emph{non-amenable} if for all $C>1$ there is $r>0$ such that \[|N_r(S)|\ge C |S| \ \text{ for all } S\subseteq X.\]
\end{definition}

We now recall Hall's marriage theorem. Given a set $Y$, $Y^{<\infty}$ denotes the set of all finite subsets of $Y$.

\begin{theorem}[Hall]
Let $X,Y$ be arbitrary sets and $\Phi \colon X \to Y^{<\infty}$. There is an injective $\varphi\colon X\to Y$ such that $\varphi(x)\in \Phi(x)$ for all $x\in X$ if and only if for all finite subset $S\subseteq X$ we have 
\[
|S|\le \Bigg|\bigcup_{x\in S}\Phi(x)\Bigg|.
\]
 \label{ThmHall}
\end{theorem}

\begin{proposition}\label{PropCoarseEmbImpliesInjCNonAme}
 Let $X$ and $Y$ be metric spaces, $X$ be u.l.f., and $f\colon X\to Y$ be a uniformly finite-to-one coarse map. If $X$ is non-amenable, then there is an injective coarse map which is close to $f$. 
\end{proposition}

\begin{proof}
As $f\colon X\to Y$ is uniformly finite-to-one, there is $m\in\N$ with $|f^{-1}(\{y\})|\leq m$ for all $y\in Y$. If $m=1$, then $f$ is already an injection and hence we can assume that $m\ge 2$. For each $r>0$, let $\Phi_r \colon X \to Y^{<\infty}$ be given by 
\[
\Phi_r(x)=f(B_X(x,r))\ \text{ for all }\ x\in X.
\] 
Notice that, as $X$ is u.l.f., each $B_X(x,r)$ is finite, so $f$ is well-defined. Notice that a map $g\colon X\to Y$ is close to $f$ if and only if there is $r>0$ such that $g(x)\in\Phi_r(x)$ for all $x\in X$. By Hall's marriage theorem (Theorem \ref{ThmHall}), this is equivalent to the existence of $r>0$ such that 
\[
|S|\le \Bigg|\bigcup_{x\in S}\Phi_r(x)\Bigg|
\]
for all finite $S\subseteq X$. Since, 
\[
\Bigg|\bigcup_{x\in S}\Phi_r(x)\Bigg|=\Bigg|\bigcup_{x\in S}f(B_X(x,r))\Bigg|=|f(N_r(S))|,
\]
it is enough notice that there is $r>0$ such that 
\begin{align*} 
|S|&\le |f(N_r(S))|
\end{align*}
for all finite $S\subseteq X$. By the non-amenability of $X$ for $C=m>1$, there is $r>0$ such that for every finite subset $S\subseteq X$, we have $m |S|\le |N_{r}(S)|$. Observing that $N_r(S)\subseteq f^{-1}(f(N_r(S)))$, our choice of $m$ gives that 
\[
|N_r(S)|\le m|f(N_r(S))|.
\]
It then follows that $|S|\le |f(N_{r}(S))|$ for all finite $S\subseteq X$ as desired.
\end{proof}

 \begin{corollary}\label{ThmBijCoarseNonAme}
 Let $X$ and $Y$ be u.l.f.\ metric spaces which are coarsely equivalent. If $X$ is non-amenable, then $X$ and $Y$ are bijectively coarsely equivalent. 
\end{corollary}

\begin{proof}
Let $f\colon X\to Y$ be a coarse equivalence. Applying Proposition \ref{PropCoarseEmbImpliesInjCNonAme} to $f$ and a coarse inverse of it, we get injective maps $f\colon X \to Y$ and $h\colon Y \to X$ which are close to $f$ and its coarse inverse, respectively. The result then follows from Lemma \ref{LemmaSchroederBersteinKonig}.
\end{proof}

\section{Isomorphisms between uniform Roe algebras of coarse disjoint unions}\label{sec:Roe}

In this section, we prove the following general result about the structure of isomorphisms between the uniform Roe algebras of coarse disjoint unions of finite metric spaces.
\begin{theorem}\label{T.|Xn|=|Yn|}
Let $t>0$. Let $X=\bigsqcup_nX_n$ and $Y=\bigsqcup_nY_n$ be u.l.f.\ metric spaces which are the coarse disjoint union of finite $t$-connected metric spaces. If $\cstu(X)$ and $\cstu(Y)$ are isomorphic, then there are cofinite subsets $N,M\subseteq\N$, a bijection $i\colon N\to M$, and a coarse equivalence $f\colon X\to Y$ such that 
\[\sum_{n\in \N\setminus N}|X_n|=\sum_{n\in \N\setminus M}|Y_n|\]
and 
\[|X_n|=|Y_{i(n)}|\ \text{ and }\ f(X_n)\subseteq Y_{i(n)}\ \text{ for all }\ n\in N.\]
\end{theorem}

Together with the results of the previous section, this implies that the coarse disjoint union of expander graphs are strongly rigid (Theorem \ref{ThmIsoURAImpliesBijCoarseEquivExp}).

Towards proving Theorem~\ref{T.|Xn|=|Yn|}, we consider the uniform Roe corona of a u.l.f.\ metric space. Precisely: given a u.l.f.\ metric space $X$, its \emph{uniform Roe corona} is the quotient
\[
\roeq(X)=\cstu(X)/\cK(\ell_2(X)).
\]
Throughout this section, 
\[
\pi\colon \cstu(X)\to \roeq(X)
\] 
denotes the canonical quotient map. 

Suppose $X=\bigsqcup_n X_n$ is the coarse disjoint union of metric spaces. We now introduce some notation which will be used for the remainder of this section. Given $A\subseteq \bbN$, we write 
\[
X_A=\bigsqcup_{n\in A} X_n
\]
and similarly for $Y$.

\begin{theorem}\label{T.Q(X).CentralProjection}
Let $t>0$ and $X=\bigsqcup_nX_n$ be the coarse disjoint union of finite $t$-connected metric spaces. A projection $p\in \roeq(X)$ is central if and only if it is of the form $\pi( \chi_{X_A})$ for some $A\subseteq \bbN$. 
\end{theorem}

The proof of Theorem \ref{T.Q(X).CentralProjection} makes use of the Higson corona. For the reader's convenience, we recall its definition. Given a u.l.f.\ metric space $(X,d_X)$, a bounded map $h\colon X\to \C$ is a \emph{Higson function} if for all $\eps,R>0$ there is a finite $F\subseteq X$ such that for all $x,y\in X\setminus F$, we have 
\[d_X(x,y)\leq R\ \text{ implies }\ |h(x)-h(y)|\leq \eps.\]
The set of all Higson functions forms a \cstar-algebra denoted by $C_h(X)$. The quotient of this algebra by $c_0(X)$ (the functions vanishing at infinity) is the \emph{Higson corona of $X$}, denoted by $C(\nu X)$, precisely: \[C(\nu X)=C_h(X)/c_0(X).\]
As $C_h(X)\subseteq \ell_\infty(X)$ and $\mathcal{K}(\ell_2(X))\cap C_h(X)=c_0(X)$, we identify $C(\nu X)$ with a \cstar-subalgebra of $ \roeq(X)$ canonically.

\begin{proof}[Proof of Theorem \ref{T.Q(X).CentralProjection}]
For the backwards direction, notice that as $X=\bigsqcup_nX_n$ is a coarse disjoint union of finite spaces, a projection of the form $\chi_{X_A}$ for $A\subseteq \bbN$ is a Higson function. It was shown in \cite[Proposition 3.6]{BaudierBragaFarahVignatiWillett2023} that $C(\nu X)=\cZ(\roeq(X))$, whence it follows that $\pi(\chi_A)$ is in $\cZ(\roeq(X))$ as desired. 

We now establish the forward direction. For that, fix $p\in \cZ(\roeq(X))=C(\nu X)$. Since $C(\nu(X))\subseteq \ell_\infty(X)/c_0(X)$, we have $p\in\ell_\infty(X)/c_0(X)$. 
Since $\ell_\infty(X)$ is a von Neumann algebra, in particular it has real rank zero, and $p$ is the image of some projection in $\ell_\infty(X)$ by the quotient map (see for example \cite[Lemma~3.1.13]{Fa:STCstar}). Fix $B\subseteq X$ such that $\pi(\chi_B)=p$. 
	
We claim that there exists $A\subseteq \bbN$ such that the symmetric difference $B\Delta X_A$ is finite. Assume otherwise. Then the set 
	\[
	A'=\{n\in\N \mid X_n\setminus B\ \text{ and }\ X_n\cap B\ \text{ are nonempty}\}
	\]
	is infinite. Since each $X_n$ is $t$-connected, for each $n\in A'$ choose $x_n\in X_n\cap B$ and $x_n'\in X_n\setminus B$ such that $d(x_n,x_n')\leq t$. For each $n\in A$, let $v_n$ be the rank-one partial isometry which sends $\delta_{x_n'}$ to $\delta_{x_n}$. So, each $v_n$ has propagation $t$, and so does the non-compact partial isometry
	\[
	v=\SOTh\sum_{n\in A'}e_{x_n,x_n'}.
	\] 
Note that $\chi_B v=v$ and $v\chi_B=0$, hence $\pi(\chi_B)$ is not central in $\roeq(X)$. 	As we can fix $A\subseteq \bbN$ such that $B\Delta X_A$ is finite, then $p=\pi( \chi_{X_A})$, as required. 
\end{proof}

To avoid overly complicated expressions, we make use of the following notation in the next lemma: if $H$ is a Hilbert space, $(p_i)_i$ is a sequence of orthogonal projections, and $A\subseteq \bbN$, we write 
\[
p_A=\SOTh\sum_{i\in A} p_i.
\]

\begin{lemma}\label{L.piqi}
	Let $H$ be a Hilbert space and $(p_n)_n$ and $(q_n)_n$ be sequences of orthogonal projections of finite rank in $\cB(H)$ such that $\SOTh\sum_n q_n=\mathrm{Id}_H$. Suppose $\eps>0$ is such that $\|p_n-q_F\|\geq \varepsilon$ for all $n\in\N$ and all $F\subseteq \N$. 
	Then there is $A\subseteq \bbN$ such that $p_A-q_B$ is not compact for all $B\subseteq \bbN$. 
\end{lemma}

\begin{proof} 
Since $\SOTh\sum_n q_n=\mathrm{Id}_H$ and each of the $p_n$'s and $q_n$'s have finite rank, we can pick an increasing sequence $(k(n))_n$ in $ \bbN$ and a partition of $\bbN$ into intervals, say $\bbN=\bigsqcup_n I(n)$, such that 
	\begin{equation}\label{eq.pn(i)}
	\|q_{I(n)}p_{k(n)}q_{I(n)}-p_{k(n)}\|<2^{-n-2}\eps
	\end{equation}
 for all $n\in\N$. In particular, compressing the expression inside the norm in line \eqref{eq.pn(i)} by $q_{I(m)}$ shows that
 \begin{equation}\label{eq.pn(i)2}
\|q_{I(m)}p_{k(n)}q_{I(m)}\|<2^{-n-2}\eps
\end{equation}
for all distinct $m,n\in\N$. We claim that 
 \[
 A=\{k(n)\mid n\geq 2\}
 \] 
 is as required. First of all, note that by \eqref{eq.pn(i)} and \eqref{eq.pn(i)2}, we have 
\begin{align*}
\|p_{k(n)}-q_{I(n)}p_A & q_{I(n)}\|\\
&\leq \|p_{k(n)}-q_{I(n)}p_{k(n)}q_{I(n)}\|+\sum_{m\neq n, m\geq 2}\|q_{I(n)}p_{k(m)}q_{I(n)}\|\\
&\leq 2^{-n-2}\eps+\sum_{m\geq 2}2^{-m-2}\eps\leq\eps/2
\end{align*}
 Assume now that $A$ does not satisfy the thesis of the lemma and pick $B\subseteq \bbN$ such that $p_A-q_B$ is compact. Using that $\|p_n-q_F\|\geq \varepsilon$ for all $n\in\N$ and all $F\subseteq \N$, we conclude that 
	\begin{align*}
	\|q_{I(n)} (p_A-q_B) q_{I(n)}\|&= 	\|q_{I(n)} p_Aq_{I(n)} -q_{B\cap I(n)}\|\\
& \geq \|p_{k(n)}-q_{B\cap I(n)}\|-\eps/2\\
 &\geq\eps/2
	\end{align*}
for all $n\in\N$. Since the $q_{I(n)}$ are orthogonal,  this contradicts the fact that $p_A-q_B$ is compact. 
\end{proof}
	
\begin{proof}[Proof of Theorem~\ref{T.|Xn|=|Yn|}]
Let $\Phi\colon \cstu(X)\to \cstu(Y)$ be an isomorphism. By \cite[Theorem 1.2]{BaudierBragaFarahKhukhroVignatiWillett2021uRaRig}, there is a coarse equivalence $f\colon X\to Y$ such that 
\begin{equation}\label{EqCoarseEquivRig}
\gamma=\inf_{x\in X}\|\Phi(\chi_{\{x\}})\delta_{f(x)}\|>0.
\end{equation}
We need to show that for some cofinite subsets $N,M\subseteq\N$ and some bijection $i\colon N\to M$, the coarse equivalence $f\colon X\to Y$ satisfies 
$\sum_{n\in \N\setminus N}|X_n|=\sum_{n\in \N\setminus M}|Y_n|$, 
$|X_n|=|Y_{i(n)}|$,   and $f(X_n)\subseteq Y_{i(n)}$ for all $n\in N$.

Since $\Phi$ is implemented by a unitary (Lemma \cite[Lemma 3.1]{SpakulaWillett2013AdvMath}), it sends compact operators to compact operators  and it therefore  induces an isomorphism between $ \roeq(X)$ and $\roeq(Y)$. It then follows from Theorem~\ref{T.Q(X).CentralProjection} that for every $ A\subseteq \bbN$ there is $ A'\subseteq \bbN$ such that $\pi(\Phi(\chi_{X_A}))=\pi(\chi_{Y_{A'}})$. In other words, we have
		\begin{equation}\label{eq.A->A'}
 \forall A\subseteq \bbN,\ \exists A'\subseteq \bbN\ \text{ such that }\ 
		\Phi(\chi_{X_A})-\chi_{Y_{A'}}\ \text{ is compact}.
		\end{equation}
Let $\eps=\gamma/4$.

\begin{claim}
There is $n_0\in \N$ such that for every finite set $F\subseteq\bbN$ with $\min F>n_0$ there is  $i(F)\subseteq\bbN$ with the property that 
		\begin{equation}
		\label{eq.nn'}
		\|\Phi(\chi_{X_F}) -\chi_{Y_{i(F)}}\|<\varepsilon.
		\end{equation}
\end{claim}

\begin{proof}
Assume not. Then we can find a sequence $(F_n)_n$ of disjoint finite subsets  of $\N$  with $\max F_n<\min F_{n+1}$ and such that for all $n\in \N$ and $G\subseteq\bbN$ we have that $\|\Phi(\chi_{X_{F_n}})-\chi_{Y_G}\|>\varepsilon$. Applying Lemma~\ref{L.piqi} with $p_n=\Phi(\chi_{X_{F_n}})$ and $q_n=\chi_{Y_n}$, we get an infinite $A\subseteq\bbN$ such that $\Phi(\chi_{X_A})-q_B$ is not compact for all infinite $B\subseteq\bbN$. This contradicts \eqref{eq.A->A'}. 
\end{proof}

Let $n_0\in \N$ be as in the claim and let \[i\colon \{F\subseteq \N\mid \min F>n_0 \text{ and }|F|<\infty\}\to \cP(\N)\] be the map given by the claim.  For simplicity, for each $n> n_0$, we write $i(n)$ for $i(\{n\})$.\footnote{With apologies to John von Neumann and any set theorists who may be reading this.}  Notice that, since $\Phi$ is an isomorphism and $\eps<1/2$, each $i(F)$  is well defined, finite, and nonempty.  Furthermore, as $\|\chi_{X_F}-\chi_{X_G}\|=1$ whenever $F\neq G$, $i$ is injective.

\begin{claim}
For all $n>n_0$ and $x\in X_n$, we have that $f(x)\in Y_{i(n)}$.
\end{claim}

\begin{proof}
Pick $m$ such that $f(x)\in Y_m$, and suppose that $m\notin i(n)$. Since $\gamma<\|\Phi(\chi_{\{x\}})\delta_{f(x)}\|$, we have that $\gamma<\|\Phi(\chi_{X_n})\chi_{Y_m}\|$. As $\|\Phi(\chi_{X_n})-\chi_{Y_{i(n)}}\|<\gamma/4$, we have that $\|\chi_{Y_{i(n)}}\chi_{Y_m}\|>0$, hence $Y_{i(n)}$ and $Y_{m}$ are not disjoint, which implies that $m\in i(n)$.
\end{proof}
Applying the  reasoning above to $Y$ and $\Phi^{-1}$ in place of $X$ and $\Phi$, we can find a natural $m_0$ and an injective function $j$ which associates to every set $G\subseteq\bbN$ with $\min G>m_0$, a set $j(G)\subseteq\bbN$ such that 
\[
\|\Phi^{-1}(\chi_{Y_G}) -\chi_{X_{j(G)}}\|<\varepsilon.
\]
Notice that if $F\subseteq\bbN$ is a finite subset such that $\min F>n_0$ and $\min i(F)>m_0$, then
\begin{align*}
\|\chi_{X_F}-\chi_{X_{j(i(F))}}\|&\leq \|\chi_{X_F}-\Phi^{-1}(\chi_{Y_{i(F)}})\|+\eps\\&=\|\Phi(\chi_{X_F})-\chi_{Y_{i(F)}}\|+\eps\\
&\leq 2\eps.
\end{align*}
As $\eps<1/2$, this  implies that $F=j(i(F))$.  Similarly, if  $G\subseteq\bbN$   is a finite subset such that $\min G>m_0$ and  $\min j(G)>n_0$, we obtain that  $G=i(j(G))$.

\begin{claim}\label{claim:almostbij}
If $n> n_0$ is such that $\min i(n)>m_0$, then $i(n)$ is a singleton.
\end{claim}
\begin{proof}
Suppose that $n$ is such that $\min i(n)>m_0$ and there are two distinct $m_1$ and $m_2$ in $i(n)$. Then $\chi_{Y_{m_1}}\chi_{Y_{i(n)}}=\chi_{Y_{m_1}}$, and so, using that $j(i(n))=n$ we have, for $k=1,2$,
\begin{align*}
\|\chi_{X_{j(m_k)}}\chi_{X_n}-\chi_{X_{j(m_k)}}\|&= \|\chi_{X_{j(m_k)}}\chi_{X_{j(i(n))}}-\chi_{X_{j(m_k)}}\|\\&=\|\Phi(\chi_{X_{j(m_k)}}\chi_{X_{j(i(n))}}-\chi_{X_{j(m_k)}})\|\\&\leq\|\chi_{Y_{m_k}}\chi_{Y_{i(n)}}-\chi_{Y_{m_k}}\|+3\eps\\
&\leq3\eps.
\end{align*}
Since $\chi_{X_{j(m_k)}}$ and $\chi_{X_n}$ are commuting projections, $\chi_{X_{j(m_k)}}\chi_{X_n}=\chi_{X_{j(m_k)}}$ for $k=1,2$. This implies $m_1=n=m_2$, contradicting the  injectivity of $j$ on sets whose minimum is above $m_0$.
\end{proof}
Let $N=\{n\in\N\mid n>n_0\text{ and } i(n)>m_0\}$ and $M=\{i(n)\mid n\in N\}$. Claim~\ref{claim:almostbij} implies that  $i$ is a bijection between $N$ and $M$.
\begin{claim}
$N$ and $M$ are cofinite.
\end{claim} 

\begin{proof}
Since $\Phi$ maps compacts to compacts, we have that $\min i(F)\to\infty$ as $\min F\to\infty$, hence $N$ is cofinite. Similarly, $\min j(G)\to\infty$ as $\min G\to\infty$. Let $m_1>m_0$ be large enough so that for all $G\subseteq\bbN$ with $\min G>m_1$ we have that $\min j(G)>n_0$. Then, by what we saw above, if $m>m_1$, we have $m=i(j(m))$. So,    $M$ is cofinite.
\end{proof}

As we have already noted, $\Phi$ is implemented by a unitary (see Lemma \cite[Lemma 3.1]{SpakulaWillett2013AdvMath}) and therefore it is  rank preserving. Thus  the defining property of $i$, see \eqref{eq.nn'}, gives that   for all  finite $F\subseteq N$ we have 
\[
|X_F|=\mathrm{rank}(\Phi(\chi_{X_F}))=\mathrm{rank}(\chi_{Y_{i(F)}})=|Y_{i(F)}|.
\] Analogously,  if $G\subseteq M$ is finite  we have that $|Y_G|=|X_{j(G)}|$. We are left to show that the sets 
\[
W\coloneqq \bigcup_{n\in\bbN\setminus N}X_n\text{ and }Z\coloneqq \bigcup_{n\in\bbN\setminus M}Y_n
\]
have the same size. Suppose this is not the case and assume that $|W|>|Z|$.

Since $\Phi$ maps compacts to compacts, we can find a large enough finite $F\subseteq Y$ such that $\|\Phi(\chi_W)\chi_F-\Phi(\chi_W)\|<1/4$. By enlarging $F$, we can assume that $Z\subseteq F$ and that if $Y_k$ intersects $F$, then $Y_k\subseteq F$. Let \[M'=\{k\in M\mid Y_k\subseteq F\}\] and notice that
\begin{equation}\label{rank1}
\rank(\chi_F)=|Z|+\rank(\chi_{Y_{M'}})=|Z|+|Y_{M'}|.
\end{equation}
Since $\chi_{Y_{M'}}\chi_F=\chi_{Y_{M'}}$ and $\|\chi_{Y_{M'}}-\Phi(\chi_{X_{j(M')}})\|<\eps$, we have that 
\[
\|\Phi(\chi_{X_{j(M')}})\chi_F-\Phi(\chi_{X_{j(M')}})\|<2\eps<\frac{1}{4}.
\]
Combining this with the fact that $\|\Phi(\chi_W)\chi_F-\Phi(\chi_W)\|<1/4$, we have   
\[
\|\Phi(\chi_W+\chi_{X_{j(M')}})\chi_F-\Phi(\chi_W+\chi_{X_{j(M')}})\|<\frac{1}{2}.
\]
 Since $j(M')\subseteq N$,   $\chi_W$ and $\chi_{X_{j(M')}}$ are orthogonal, and so $\Phi(\chi_W+\chi_{X_{j(M')}})$ is a projection of rank $|W|+|X_{j(M')}|$. By usual linear algebra arguments, 
\begin{equation}\label{rank2}
|W|+|X_{j(M')}|=\rank (\Phi(\chi_W+\chi_{X_{j(M')}}))\leq\rank(\chi_F).
\end{equation}
 Putting \eqref{rank1} and \eqref{rank2} together and using that $|X_{j(M')}|=|Y_{M'}|$, we get that 
\begin{align*}
\rank(\chi_F)&=|Z|+|Y_{M'}|<|W|+|Y_{M'}|=|W|+|X_{j(M')}|\leq\rank(\chi_F).
\end{align*}
This is a contradiction, and therefore $|W|\leq |Z|$. The same exact argument proves that $|Z|\leq |W|$ and this finishes the proof.
	\end{proof}

\begin{proof}[Proof of Theorem \ref{ThmIsoURAImpliesBijCoarseEquivExp}]
 By Theorem \ref{T.|Xn|=|Yn|}, there is a coarse equivalence $f\colon X\to Y$ which satisfies part \eqref{ThmBijCorExpItem2} of Theorem \ref{ThmBijCorExp}. By the equivalences in Theorem \ref{ThmBijCorExp}, it follows that there is a bijective coarse equivalence $g\colon X\to Y$.
\end{proof}
 
 \begin{acknowledgments}
This paper was written under the auspices of the American Institute of Mathematics (AIM) SQuaREs program as part of the `Expanders, ghosts, and Roe algebras' SQuaRE project. F.\ B.\ was partially supported by the US National Science Foundation under the grants DMS-1800322 and DMS-2055604. B.\ M.\ B.  was partially supported by FAPERJ (Proc. E-26/200.167/2023) and by CNPq (Proc. 303571/2022-5). I.\ F.\ is partially supported by NSERC. A.\ V.\ is supported by an `Emergence en Recherche' IdeX grant from Universit\'e Paris Cit\'e and an ANR grant (ANR-17-CE40-0026). R.\ W.\ was partially supported by the US National Science Foundation under the grants DMS-1901522 and DMS-2247968.
\end{acknowledgments}

\newcommand{\etalchar}[1]{$^{#1}$}
\providecommand{\bysame}{\leavevmode\hbox to3em{\hrulefill}\thinspace}
\providecommand{\MR}{\relax\ifhmode\unskip\space\fi MR }
\providecommand{\MRhref}[2]{%
  \href{http://www.ams.org/mathscinet-getitem?mr=#1}{#2}
}
\providecommand{\href}[2]{#2}

\end{document}